\documentclass[11pt,a4paper]{article}
\usepackage[top=3cm, bottom=3cm, left=3cm, right=3cm]{geometry}
\usepackage{tabu}
\usepackage{breqn}
\usepackage[hang,flushmargin]{footmisc} 

\usepackage{amsmath,amsthm,amssymb,graphicx, multicol, array}
\usepackage{enumerate}
\usepackage{enumitem}
\usepackage{hyperref}
\usepackage{setspace}
\usepackage{xcolor}
\usepackage{currfile}

\usepackage{tabto}

\DeclareMathOperator{\ord}{ord}
\DeclareMathOperator{\ind}{ind}

\newtheorem{thm}{Theorem}[section]
\newtheorem{lem}{Lemma}[section]

\newtheorem{cor}{Corollary}[section]
\newtheorem{dfn}{Definition}[section]

\newcommand{\N}{\mathbb{N}}
\newcommand{\Z}{\mathbb{Z}}
\newcommand{\Q}{\mathbb{Q}}

\newcommand{\C}{\mathbb{C}}
\newcommand{\F}{\mathbb{F}}

\usepackage{fancyhdr}
\title{Simultaneous Elements Of Prescribed Multiplicative Orders}
\date{}
\author{N. A. Carella}

\begin{document}
\thispagestyle{empty}
\date{}

\maketitle
\textbf{\textit{Abstract}:} Let $u\ne \pm 1$, and $v\ne \pm 1$ be a pair of fixed relatively prime squarefree integers, and let $d\geq 1$, and $e \geq1$ be a pair of fixed integers. It is shown that there are infinitely many primes $p\geq 2$ such that $u$ and $v$ have simultaneous prescribed multiplicative orders $\ord_pu=(p-1)/d$ and $\ord_pv=(p-1)/e$ respectively, unconditionally. In particular, a squarefree odd integer $u>2$ and $v=2$ are simultaneous primitive roots and quadratic residues (or quadratic nonresidues) modulo $p$ for infinitely many primes $p$, unconditionally. \let\thefootnote\relax\footnote{ \today \date{} 
}

\tableofcontents
\section{Introduction} \label{S8888}
The earliest study of admissible integers $k$-tuples $u_1,u_2,\ldots,u_k \in \Z$ of simultaneous primitive roots modulo a prime $p$ seems to be the conditional result in \cite{MK76}. Much more general results for admissible rationals $k$-tuples $u_1,u_2,\ldots,u_k\in \Q$ of simultaneous elements of independent (or pseudo independent) multiplicative orders modulo a prime $p\geq2$ are considered in \cite{PS09}, and \cite{JO19}. An important part of this problem is the characterization of admissible rationals $k$-tuples. There are various partial characterizations of admissible rationals $k$-tuples. For example, an important criterion states that a rationals $k$-tuple is admissible if and only if
\begin{equation}\label{eq8888.000}
u_1^{e_1}u_2^{e_2}\cdots u_k^{e_k}\ne -1,
\end{equation}   
where $e_1,e_2,\ldots,e_k\in \Z$ are integers. A more general characterization and the proof appears in \cite[Proposition 14]{PS09}, \cite[Lemma 5.1]{JO19}. Other ad hoc techniques are explained in \cite{JP21}. The characterization of admissible triple of rational numbers $a,b,c\in \Q-\{-1,0,1\}$ with simultaneous equal multiplicative orders
\begin{equation}\label{eq8888.005}
\ord_p a =\ord_p b= \ord_p c,
\end{equation}
known as the Schinzel-Wojcik problem, is an open problem, see \cite[p.\ 2]{PS09}, and \cite{FM18}. The specific relationship between the orders $\ord_p u\mid d$ and $\ord_p v\mid d$ of a pair of integers $u,v \ne\pm1$ and some divisor $d\mid p-1$, and the generalization to number fields, and abelian varieties are studied in several papers as \cite{CS97}, \cite{BS10}, et cetera, and counterexamples are produced in \cite{CS97} and \cite{LS06}.

\begin{dfn}\label{dfn8888.000} {\normalfont A $k$-tuple $u_1,u_2,\ldots, u_k\ne\pm1$ of rational numbers is called an \textit{admissible} $k$-tuple if the product is multiplicatively independent over the rational numbers:  
\begin{equation}\label{eq8888.010}
u_1^{e_1}u_2^{e_2}\cdots u_k^{e_k}\ne 1,
\end{equation}  
for any list of integer exponents $e_1,e_2,\ldots, e_k \in \Z^{\times}$.
}
\end{dfn}
Relatively prime $k$-tuples, and relatively prime and squarefree $k$-tuples are automatically multiplicatively independent over the rational numbers. However, squarefree $k$-tuples are not necessarily multiplicatively independent over the rational numbers, for example, $3,5,15$. Any $k$-tuple of integers that satisfies the Lang-Waldschmid conjecture is an admissible $k$-tuple. Specifically,
\begin{equation}\label{eq8888.020}
\left|u_1^{e_1}u_2^{e_2}\cdots u_k^{e_k} -1\right | \geq\frac{C(\varepsilon)^kE}{\left |u_1u_2\cdots u_k\cdot e_1e_2\cdots ae_k\right |^{1+\varepsilon}},
\end{equation}    
where $C(\varepsilon)>0$ is a constant, $E=\max\{|e_i|\}$, and $\varepsilon>0$ is a small number, confer \cite[Conjecture 2.5]{WM03} for details.   
\begin{dfn}\label{dfn8888.005} {\normalfont Fix an admissible $k$-tuple $u_1,u_2,\ldots, u_k\ne\pm1$ of rational numbers. The elements $u_1,u_2,\ldots,u_k \in \F_p$ are said to be a simultaneous $k$-tuple of equal multiplicative orders modulo a prime $p\geq 2$, if $\ord_pu_1\mid p-1$, and  
\begin{equation}\label{eq8888.025}
\ord_pu_1 =\ord_pu_2=\cdots =\ord_pu_k,
\end{equation}  
infinitely often as $p\to \infty$.}
\end{dfn}

\begin{dfn}\label{dfn8888.010} {\normalfont Fix an admissible $k$-tuple $u_1,u_2,\ldots, u_k\ne\pm1$ of rational numbers. The elements $u_1,u_2,\ldots,u_k \in \F_p$ are said to be a simultaneous $k$-tuple of decreasing multiplicative orders modulo a prime $p$, if $\ord_pu_i\mid p-1$, and
\begin{equation}\label{eq8888.030}
\ord_pu_1 >\ord_pu_2>\cdots> \ord_pu_k,
\end{equation}
infinitely often as $p\to \infty$. 
}
\end{dfn}

\begin{dfn}\label{dfn8888.015} {\normalfont Fix an admissible $k$-tuple $u_1,u_2,\ldots, u_k\ne\pm1$ of rational numbers, and fix an index integers $k$-tuple $d_1,d_2,\ldots, d_k\geq 1$. The elements $u_1,u_2,\ldots,u_k \in \F_p$ are called a simultaneous $k$-tuple of prescribed multiplicative orders modulo a prime $p$, if  
\begin{equation}\label{eq8888.035}
\ord_pu_1 =(p-1)/d_1, \quad\ord_pu_2=(p-1)/d_2, \quad\cdots\quad \ord_pu_k=(p-1)/d_k,
\end{equation}
infinitely often as $p\to \infty$. 
}
\end{dfn}
Conditional on the GRH and or the $k$-tuple primes conjecture, several results for the existence and the densities of simultaneous rationals $k$-tuples have been proved in the literature, see \cite{PS09}, \cite{FM18}, and \cite{JO19}. This note studies the unconditional asymptotic formulas for the number of primes with simultaneous elements of prescribed multiplicative orders.

\begin{thm} \label{thm8888.002} Fix a pair of relatively prime squarefree $u\ne \pm1$, and $v\ne \pm1$ rational numbers , and fix a pair of integers $d\geq 1$, and $e\geq 1$. If $x\geq 1$ is a sufficiently large number, and the indices $d, e \ll (\log x)^B$, with $B\geq 0$, then, the number of primes $p\in [x,2x]$ with simultaneous elements $u$ and $v$ of prescribed multiplicative orders $\ord_pu=(p-1)/d$ and $\ord_pv = (p-1)/e$ modulo $p\geq 2$, has the asymptotic lower bound
\begin{equation}\label{eq8888.040}
R_2(x,u,v)\gg\frac{x}{(\log x)^{4B+1}(\log \log x)^2}
\end{equation}
as $x \to \infty$, unconditionally.
\end{thm}
In general, the number of primes $p\in [x,2x]$ such that a fixed admissible $k$-tuples of rational numbers $u_1,u_2,\ldots, u_k\ne\pm1$ has simultaneous prescribed multiplicative orders
\begin{equation}\label{eq8888.045}
\ord_pu_1=(p-1)/d_1, \quad \ord_pu_2=(p-1)/d_2, \quad \ldots, \quad \ord_pu_k=(p-1)/d_k,
\end{equation}
where $d_i\ll (\log x)^B$ are fixed indices, and $B\geq 0$, has the lower bound
\begin{equation}\label{eq8888.050}
R_k(x,u,v)\gg\frac{x}{(\log x)^{2Bk+1}(\log \log x)^k}
\end{equation}
as $x \to \infty$, unconditionally. \\

The maximal number of simultaneous primitive roots is bounded by $k=O\left ( \log p\right )$, see \cite[Section 14]{CN19}. The same upper bound is expected to hold for any combination of simultaneous $k$-tuple prescribed of multiplicative orders, but it has not been verified yet. The average multiplicative order of a fixed rational number $u\ne\pm1$ has the asymptotic lower bound
\begin{equation}\label{eq8888.053}
T_u(x)=\frac{1}{x}\sum_{\substack{n\leq x\\\gcd(u,n)=1}} \ord_nu\gg\frac{x}{\log x}e^{c(\log \log \log x)^{3/2}},
\end{equation}
where $c>0$ is a constant, the fine details are given in \cite{KP13}. Other information on the average multiplicative orders in finite cyclic groups are given in \cite{LF05}, and the literature. \\

A special case illustrates the existence of simultaneously prescribed primitive roots and quadratic residues (or quadratic nonresidues) in the prime finite field $\F_p$ for infinitely many primes $p$.
\begin{cor}\label{cor8888.104} Let $u\geq3$ be fixed a squarefree odd integer and let $v=2$. If $x\geq 1$ is a large number, then elements $u, 2 \in \F_p$ have multiplicative orders $\ord_pu=p-1$ and $\ord_p2=(p-1)/2$, simultaneously. Moreover, the number of such primes has the asymptotic lower bound
\begin{equation}\label{eq8888.055}
R_2(x,u,2)\gg\frac{x}{(\log x)(\log \log x)^2}
\end{equation}
as $x \to \infty$, unconditionally.
\end{cor}

The unconditional number of primes with simultaneous admissible triple of rational numbers $a,b,c\in \Q-\{-1,0,1\}$ of equal multiplicative orders 
\begin{equation}\label{eq8888.060}
\ord_p a =\ord_p b= \ord_p c,
\end{equation}
such that $\ord_p a =(p-1)/d$, with $d\ll (\log x)^B$, and $B\geq 0$, has almost identical analysis as the proof of Theorem \ref{thm8888.002}. In fact, any permutation of equality or inequality between the multiplicative orders can be produced by selecting any small fixed indices $d,e,f\geq 1$ to prescribe the multiplicative orders $\ord_pa=(p-1)/d$, $\ord_pb=(p-1)/e$, and $\ord_pc=(p-1)/f$ respectively. \\

Some of the foundation works on the calculations of the implied constants in \eqref{eq8888.010} appear in \cite{PS09}, \cite{FM18}, \cite{JO19}, et alii. The proof of Theorem \ref{thm8888.002} appears in Section \ref{S3388}, this result is completely unconditional. The other sections cover foundational and auxiliary materials.

\section{Divisors Free Characteristic Function}\label{S3000}

\begin{dfn} \label{dfn3000.005}{\normalfont
The \textit{multiplicative order} of an element in the cyclic group $\mathbb{F}_p^\times$ is defined by $\ord_p (v)=\min\{k:v^k\equiv 1 \bmod p\}$. \textit{Primitive} elements in this cyclic group have order $p-1=\#G$.
}
\end{dfn}
\begin{dfn}\label{dfn3000.010} {\normalfont Let $p\geq 2$ be a prime, and let $\F_p$ be the prime field of characteristic $p$. If $d\mid p-1$ is a small divisor, the multiplicative subgroup of $d$-powers, and the subgroup \textit{index} are defined by $\F_p^d=\{\alpha^d: \alpha\in \F_p^{\times}\}$, and $[\F_p^{\times}:\F_p^d]=d$ respectively.   
}
\end{dfn}
\begin{dfn}\label{dfn3000.015} {\normalfont Fix an integer $d\geq1$, and a rational number $u \in\Q^{\times}$. An element $u\in \F_p$ has \textit{index} $d=\ind_pu$ modulo a prime $p$ if and only if $\ind_pu=(p-1)/\ord_pu$. In particular, if $d\mid p-1$, then the prime finite field $\F_p$ contains $\varphi(d)$ elements of index $d\geq1$. 
}
\end{dfn}

Each element $u\in \F_p$ of index $d$ is a $d$-power, but not conversely. 
A new \textit{divisor-free} representation of the characteristic function of primitive element is introduced here. This representation can overcomes some of the limitations of the standard \textit{divisor-dependent} representation of the characteristic function
\begin{eqnarray}\label{eq3000.005}
\Psi (u)&=&\frac{\varphi (p-1)}{p-1}\sum _{d \mid p-1} \frac{\mu (d)}{\varphi (d)}\sum _{\ord(\chi ) = d} \chi (u)\\
&=&
\left \{\begin{array}{ll}
1 & \text{ if } \ord_p (u)=p-1,  \\
0 & \text{ if } \ord_p (u)\neq p-1, \\
\end{array} \right .\nonumber
\end{eqnarray}
of primitive roots. The works in \cite{DH37}, and \cite{WR01} attribute this formula to Vinogradov. The proof and other details on this representation of the characteristic function of primitive roots are given in \cite[p. 863]{ES57}, \cite[p.\ 258]{LN97}, \cite[p.\ 18]{MP07}. Equation \eqref{eq3000.005} detects the multiplicative order \(\ord_p (u)\) of the element \(u\in \mathbb{F}_p\) by means of the divisors of the totient \(p-1\). In contrast, the \textit{divisors-free} representation of the characteristic function in \eqref{eq3000.105} detects the multiplicative order \(\ord_p(u) \geq 1\) of the element \(u\in \mathbb{F}_p\) by means of the solutions of the equation \(\tau ^n-u=0\) in \(\mathbb{F}_p\), where \(u,\tau\) are constants, and \(1\leq n<p-1, \gcd (n,p-1)=1,\) is a variable. 

\begin{lem} \label{lem3000.103}
Let \(p\geq 2\) be a prime, and let \(\tau\) be a primitive root mod \(p\). Let $\psi (z)=e^{i 2\pi z/p }\neq 1$ be a nonprincipal additive character of order \(\ord \psi =p\). If \(u\in\mathbb{F}_p\) is a nonzero element, then,
\begin{eqnarray}\label{eq3000.105}
\Psi (u)&=&\sum _{\gcd (n,p-1)=1} \frac{1}{p}\sum _{0\leq k\leq p-1} \psi \left ((\tau ^n-u)k\right)\\
&=&\left \{
\begin{array}{ll}
1 & \text{ if } \ord_p(u)=p-1,  \\
0 & \text{ if } \ord_p(u)\neq p-1. \\
\end{array} \right . \nonumber
\end{eqnarray}
\end{lem}

\begin{proof} As the index \(n\geq 1\) ranges over the integers relatively prime to \(p-1\), the element \(\tau ^n\in \mathbb{F}_p\) ranges over the primitive roots \(\text{mod } p\). Ergo, the equation
\begin{equation}\label{eq3000.107}
\tau ^n- u=0
\end{equation} has a solution if and only if the fixed element \(u\in \mathbb{F}_p\) is a primitive root. Next, replace \(\psi (z)=e^{i 2\pi  k z/p }\) to obtain
\begin{eqnarray}
\Psi(u)&=&\sum_{\gcd (n,p-1)=1} \frac{1}{p}\sum_{0\leq k\leq p-1} e^{i 2\pi  (\tau ^n-u)k/p }\\
&=&\left \{
\begin{array}{ll}
1 & \text{ if } \ord_p (u)=p-1,  \\
0 & \text{ if } \ord_p (u)\neq p-1. \\
\end{array} \right.\nonumber
\end{eqnarray}
This follows from the geometric series identity $\sum_{0\leq k\leq N-1} w^{ k }=(w^N-1)/(w-1)$ with $w \ne 1$, applied to the inner sum.
\end{proof}
Let $d\mid p-1$. A new representation of the indicator function for $d$-power \(v\in \mathbb{F}_p\) or elements of order $\ord_p(v)=(p-1)/d$ is consider below.
\begin{lem} \label{lem3000.114}
Let \(p\geq 2\) be a prime, and let \(\tau\) be a primitive root mod \(p\). Let $\psi (z)=e^{i 2\pi z/p }\neq 1$ be a nonprincipal additive character of order \(\ord \psi =p\). If \(u\in\mathbb{F}_p\) is a $d$-power, then,
\begin{eqnarray}\label{eq3000.117}
\Psi (u,d)&=&\sum _{\gcd (n,(p-1)/d)=1} \frac{1}{p}\sum _{0\leq k\leq (p-1)/d} \psi \left ((\tau ^{dn}-u)k\right)\\
&=&\left \{
\begin{array}{ll}
1 & \text{ if } \ord_p(u)=(p-1)/d,  \\
0 & \text{ if } \ord_p(u)\neq (p-1)/d. \\
\end{array} \right .\nonumber
\end{eqnarray}
\end{lem}
\begin{proof}Similar to the proof of Lemma \ref{lem3000.103}, mutatis mutandis.
\end{proof}

\section{Finite Summation Kernels} \label{S3366}
Let $f: \C \longrightarrow \C$ be a function, and let $q \in \N$ be a large integer. The finite Fourier transform 
\begin{equation} \label{eq3366.102}
\hat{f}(t)=\frac{1}{q} \sum_{0 \leq s\leq q-1} e^{i \pi st/q}
\end{equation}
and its inverse are used here to derive a summation kernel function, which is almost identical to the Dirichlet kernel. 

\begin{dfn} \label{dfn3366.104} {\normalfont Let $ p$ and $ q $ be primes, and let $\omega=e^{i 2 \pi/q}$, and $\zeta=e^{i 2 \pi/p}$ be roots of unity. The \textit{finite summation kernel} is defined by the finite Fourier transform identity
\begin{equation} \label{eq3366.104}
\mathcal{K}(f(n))=\frac{1}{q} \sum_{0 \leq t\leq q-1,}  \sum_{0 \leq s\leq p-1} \omega^{t(n-s)}f(s)=f(n).
\end{equation}
} 
\end{dfn}
This simple identity is very effective in computing upper bounds of some exponential sums
\begin{equation}\label{eq3366.106}
 \sum_{ n \leq x}  f(n)= \sum_{ n \leq x}  \mathcal{K}(f(n)),
\end{equation}
where $x  \leq p < q$. This technique generalizes the sum of resolvents method used in \cite{ML72}. Here, it is reformulated as a finite Fourier transform method, which is applicable to a wide range of functions.

\begin{lem}   \label{lem3366.100}  Let \(p\geq 2\) and $q=p+o(p)$ be large primes. Let $\omega=e^{i2 \pi/q} $ be a $q$th root of unity, and let $t \in [1, p-1]$. Then,
\begin{enumerate} [font=\normalfont, label=(\roman*)]
\item $\displaystyle  
	 \sum_{n \leq p-1} \omega^{tn} = \frac{\omega^{t}-\omega^{tp}}{1-\omega^{t}},
$
\item $\displaystyle
	\left |  \sum_{n \leq p-1} \omega^{tn}  \right |\leq \frac{2q }{\pi t}.
$ 
 
\end{enumerate}
\end{lem} 

\begin{proof} (i) Use the geometric series to compute this simple exponential sum as
	\begin{eqnarray} \label{eq3366.102}
	 \sum_{n \leq p-1} \omega^{tn}
	&=& \frac{\omega^{t}-\omega^{tp}}{1-\omega^{t}} \nonumber.
	\end{eqnarray} 
(ii) Observe that the parameters $q=p+o(p)$ is prime, $\omega=e^{i2 \pi/q}$, the integers $t \in [1, p-1]$, and $d \leq p-1<q-1$. This data implies that $\pi t/q\ne k \pi $ with $k \in \mathbb{Z}$, so the sine function $\sin(\pi t/q)\ne 0$ is well defined. Using standard manipulations, and $z/2 \leq \sin(z) <z$ for $0<|z|<\pi/2$, the last expression becomes
	\begin{equation}
	\left |\frac{\omega^{t}-\omega^{tp}}{1-\omega^{t}} \right |\leq 	\left | \frac{2}{\sin( \pi t/ q)} \right | 
	\leq \frac{2q}{\pi t}.
	\end{equation}
	
\end{proof}

\begin{lem}   \label{lem3366.200}  Let \(p\geq 2\) and $q=p+o(p)$ be large primes, and let $\omega=e^{i2 \pi/q} $ be a $q$th root of unity. Then,
\begin{enumerate} [font=\normalfont, label=(\roman*)]
\item $\displaystyle
	 \sum_{\gcd(n,(p-1)/d)=1} \omega^{tn} = \sum_{r \mid (p-1)/d} \mu(d) \frac{\omega^{rt}-\omega^{drt((p-1)/r(d+1))}}{1-\omega^{rdt}},
$
\item $\displaystyle
	\left | \sum_{\gcd(n,p-1)=1} \omega^{tn}  \right |\leq \frac{4q \log \log p}{\pi t},
$ 
 \end{enumerate}
where $\mu(k)$ is the Mobius function, for any fixed pair $d \mid p-1$ and $t \in [1, p-1]$.
\end{lem} 

\begin{proof} (i) Use the inclusion exclusion principle to rewrite the exponential sum as
	\begin{eqnarray} \label{eq3366.204}
	\sum_{\gcd(n,p-1)=1} \omega^{tn}&=& \sum_{n \leq p-1} \omega^{tn}  \sum_{\substack{d \mid p-1 \\ d \mid n}}\mu(d)  \nonumber \\
	&=& \sum_{d \mid p-1} \mu(d) \sum_{\substack{n \leq p-1 \\ d \mid n}} \omega^{tn}\nonumber \\
	& =&\sum_{d\mid p-1} \mu(d) \sum_{m \leq (p-1)/ d} \omega^{dtm} \\
	&=& \sum_{d \mid p-1} \mu(d) \frac{\omega^{dt}-\omega^{dt((p-1)/d+1)}}{1-\omega^{dt}} \nonumber.
	\end{eqnarray} 
(ii) Observe that the parameters $q=p+o(p)$ is prime, $\omega=e^{i2 \pi/q}$, the integers $t \in [1, p-1]$, and $d \leq p-1<q-1$. This data implies that $\pi dt/q\ne k \pi $ with $k \in \mathbb{Z}$, so the sine function $\sin(\pi dt/q)\ne 0$ is well defined. Using standard manipulations, and $z/2 \leq \sin(z) <z$ for $0<|z|<\pi/2$, the last expression becomes
	\begin{equation}
	\left |\frac{\omega^{dt}-\omega^{dtp}}{1-\omega^{dt}} \right |\leq 	\left | \frac{2}{\sin( \pi dt/ q)} \right | 
	\leq \frac{2q}{\pi dt}
	\end{equation}
	for $1 \leq d \leq p-1$. Finally, the upper bound is
	\begin{eqnarray}
	\left|   \sum_{d \mid p-1} \mu(d) \frac{\omega^{dt}-\omega^{dt((p-1)/d+1)}}{1-\omega^{dt}} \right| 
	&\leq&\frac{2q}{\pi t} \sum_{d \mid p-1} \frac{1}{d} \\
	&\leq& \frac{4q \log \log p}{\pi t} \nonumber.
	\end{eqnarray}
The last inequality uses the elementary estimate $ \sum_{d \mid n} d^{-1} \leq 2 \log \log n$.
\end{proof}

\section{Gaussian Sums, And Weil Sums} \label{S4466}
\begin{thm}   {\normalfont (Gauss sums)} \label{thm4466.400}  Let $p\geq 2$ and $q \geq 2$ be large primes. Let $\tau$ be a primitive root modulo $p$. If $\chi(t)=e^{i2 \pi t/q} $ and  $\psi(t)=e^{i2\pi  \tau^t/p}$ are a pair of characters, then, the Gaussian sum has the upper bound
	\begin{equation} \label{eq4466.400}
	\left |\sum_{1 \leq t \leq q-1}    \chi(t) \psi(t) \right | \leq 2 q^{1/2} \log q.
	\end{equation}

\end{thm} 
\begin{thm}   {\normalfont (Weil sums)} \label{thm4466.420}  Let $p\geq 2$ and $q \geq 2$ be large primes, and let $f(t)$ be a powerfree polynomial of degree $\deg f=d\geq 1$. If $\chi(t)=e^{i2 \pi t/q} $ and  $\psi(t)=e^{i2\pi  f(t)/p}$ are a pair of characters. Then, the Weil sum has the upper bound
	\begin{equation} \label{eq4466.430}
	\left |\sum_{1 \leq t \leq q-1}    \chi(t) \psi(f(t)) \right | \leq 2 d q^{1/2} \log q.
	\end{equation}

\end{thm} 
\begin{thm}   \label{thm4466.430}  Let $p\geq 2$ and $q \geq 2$ be large primes. Let $\tau$ be a primitive root modulo $p$, and let $\kappa=\tau^d$ be an element of large multiplicative order. If $\chi(t)=e^{i2 \pi t/q} $ and  $\psi(t)=e^{i2\pi  t/p}$ are a pair of characters. Then, the exponential sum has the upper bound
	\begin{equation} \label{eq4466.430}
	\left |\sum_{1 \leq t \leq q-1}    \chi(t) \psi(\tau^{dt}) \right | \leq 2 d q^{1/2} \log q.
	\end{equation}

\end{thm} 
\begin{proof} Use the change of variable $z=\tau^t$ to rewrite the exponential sum as a Weil sum with a polynomial $f(z)=z^d$ of degree $d$.
\end{proof}

\section{Incomplete And Complete Exponential Sums}
Two applications of the generalizes the sum of resolvents method used in \cite{ML72}, and \cite{SR73}, to estimate exponential sums are illustrated here. The first application is a nonlinear counterpart of the Polya-Vinogradov inequality
\begin{equation}\label{eq3344.108}
 \sum_{ n \leq x}  \chi(n)\leq 2p^{1/2}\log p
\end{equation}
for nonprincipal character $\chi\ne1$ modulo $p$.

\begin{thm}  \label{thm3344.100} Let \(p\geq 2\) be a large prime, and let \(\kappa \in \mathbb{F}_p\) be an element of large multiplicative order $\ord_p(\kappa) \mid p-1$. Then, for any fixed integer $a \in [1, p-1]$, and $x\leq p-1$,
\begin{equation}\label{eq3344.200}
 \sum_{ n \leq x}  e^{i2\pi a \kappa^{n}/p} \ll p^{1/2}  \log^3 p.
\end{equation}
\end{thm}

\begin{proof} Let $q=p+o(p)$ be a large prime, and write $f(n)=e^{i 2 \pi a\tau^{dn} /p}$, where $\tau $ is a primitive root modulo $p$, and $\kappa=\tau^d$ has large multiplicative order modulo $p$ modulo $p$. Applying the finite summation kernel in Definition \ref{dfn3366.104}, yields
\begin{eqnarray} \label{eq3344.212}
R(d,x)&=& \sum_{ n \leq x}  e^{i2\pi a \tau^{dn}/p}\\
&=& \sum_{ n \leq x}\frac{1}{q} \sum_{0 \leq t\leq q-1,}  \sum_{1 \leq s\leq p-1} \omega^{t(n-s)}e^{i2\pi a \tau^{ds}/p} \nonumber.
\end{eqnarray}
Use the Weil sum upper bound, see Theorem \ref{thm4466.430}, to show that the value $t=0$ contributes
\begin{eqnarray} \label{eq3344.214}
\frac{1}{q} \sum_{ n \leq x,} \sum_{1 \leq s\leq p-1} e^{i2\pi a \tau^{ds}/p}
&=&\frac{x}{q} \sum_{1 \leq s\leq p-1} e^{i2\pi a \tau^{ds}/p}\\
&=&\frac{x}{q} \sum_{1 \leq z\leq p-1} e^{i2\pi a z^{d}/p}\nonumber\\
&\leq&\frac{x}{q} \cdot (2\log p)p^{1/2}  \nonumber\\
&\leq&\frac{4x\log p}{p^{1/2}} \nonumber,
\end{eqnarray}
where $1/q\leq 2/p$. Replacing \eqref{eq3344.214} into \eqref{eq3344.212}, and rearranging it, yield
\begin{eqnarray} \label{eq3344.216}
R(d,x)&=&\sum_{ n \leq x}  e^{i2\pi a \tau^{dn}/p}\\
&=&\frac{1}{q} \sum_{ n \leq x,} \sum_{1 \leq t\leq q-1,}  \sum_{1 \leq s\leq p-1} \omega^{t(n-s)}e^{i2\pi a \tau^{ds}/p}+O\left (\frac{4x\log p}{p^{1/2}}\right ) \nonumber\\
&=&\frac{1}{q}  \sum_{1 \leq t\leq q-1}  \left (\sum_{1 \leq s\leq p-1} \omega^{-ts}e^{i2\pi a \tau^{ds}/p} \right ) \left (\sum_{ n \leq x}\omega^{tn} \right )+O\left (\frac{4x\log p}{p^{1/2}}\right )\nonumber.
\end{eqnarray}
Taking absolute value, and applying Lemma \ref{lem3366.100} to the inner sum, and Theorem \ref{thm4466.430} to the middle sum, yield
\begin{eqnarray} \label{eq3344.218}
|R(d,x)|&=&\left | \sum_{ n \leq x}  e^{i2\pi a \tau^{dn}/p} \right |\\
&\leq&\frac{1}{q}  \sum_{1 \leq t\leq q-1} \left | \sum_{1 \leq s\leq p-1} \omega^{-ts}e^{i2\pi a \tau^{ds}/p} \right | \cdot  \left | \sum_{ n \leq x}\omega^{tn} \right |+ O\left (\frac{4x\log p}{p^{1/2}}\right )\nonumber \\
&\ll&\frac{1}{q}  \sum_{1 \leq t\leq q-1} \left ( 2p^{1/2} \log^2 p \right ) \cdot  \left ( \frac{2q}{\pi t} \right )+\frac{4x\log p}{p^{1/2}}\nonumber\\
&\ll& p^{1/2} \log^3 p+\frac{4x\log p}{p^{1/2}}\nonumber .
\end{eqnarray}
The last summation in (\ref{eq3344.218}) uses the estimate 
\begin{equation} \label{eq3344.220} 
\sum_{1 \leq t\leq q-1}\frac{1}{t}\ll \log q\ll \log p
\end{equation} 
since $q=p+o(p)$.
\end{proof}
This result is nontrivial for $x\geq p^{1/2+\delta}$, and elements of large multiplicative orders $\ord_p\kappa \geq  p^{1/2+\delta}$, where $\delta>0$. A similar upper bound for composite moduli $p=m$ is also proved in \cite[Equation (2.29)]{ML72}. \\

The second application is a complete exponential sum version of the previous one, but restricted to relatively prime arguments. 

\begin{thm}  \label{thm3344.500}  Let \(p\geq 2\) be a large prime, let $\tau $ be a primitive root modulo $p$, and let $\kappa=\tau^d$ be an element of large multiplicative order modulo $p$. Then,
\begin{equation}\label{eq3344.500}
	 \sum_{\gcd(n,(p-1)/d)=1} e^{i2\pi a \tau^{dn}/p} \ll  p^{1-\varepsilon} 
\end{equation} 
for any fixed integer $a \in [1, p-1]$, and any arbitrary small number $\varepsilon \in 
(0, 1/2)$. 
\end{thm} 

\begin{proof} Let $q=p+o(p)$ be a large prime, and write $f(n)=e^{i 2 \pi a\tau^{dn} /p}$, where $\tau$ is a primitive root modulo $p$. Start with the representation
\begin{equation} \label{eq3344.502}
\sum_{ \gcd(n,(p-1)/d)=1}  e^{\frac{i2\pi a \tau^{dn}}{p}}= \sum_{ \gcd(n,(p-1)/d)=1}\frac{1}{q} \sum_{0 \leq t\leq q-1,}  \sum_{1 \leq s\leq p-1} \omega^{t(n-s)}e^{\frac{i2\pi a \tau^{ds}}{p}} ,
\end{equation}
see Definition \ref{dfn3366.104}. Use the inclusion exclusion principle to rewrite the exponential sum as
\begin{equation}\label{eq3344.504}
\sum_{\gcd(n,(p-1)/d)=1} e^{ \frac{i2\pi a \tau^{dn}}{p}} 
= \sum_{ n \leq (p-1)/d}\frac{1}{q} \sum_{0 \leq t\leq q-1,}  \sum_{1 \leq s\leq p-1} \omega^{t(n-s)}e^{\frac{i2\pi a \tau^{ds}}{p}} \sum_{\substack{r \mid (p-1)/d \\ r \mid n}}\mu(r)   .
\end{equation} 
The value $t=0$ contributes 
\begin{eqnarray}\label{eq3344.506}
T_0(d,p)&=& \sum_{ n \leq (p-1)/d}\frac{1}{q} \sum_{1 \leq s\leq p-1} e^{\frac{i2\pi a \tau^{ds}}{p}} \sum_{\substack{r \mid (p-1)/d \\ r \mid n}}\mu(r)\nonumber\\ 
 &\leq &\frac{1}{q}\sum_{ r \mid (p-1)/d} \left |\sum_{1 \leq s\leq p-1} e^{\frac{i2\pi a \tau^{ds}}{p}}\right | \sum_{ m \leq (p-1)/rd}1 \\
&\leq &\frac{1}{q}\frac{p-1}{d} \left |\sum_{1 \leq s\leq p-1} e^{\frac{i2\pi a \tau^{ds}}{p}}\right | \sum_{ r \mid (p-1)/d}\frac{1}{r}\nonumber\\
&\leq &\frac{2p^{1/2}\log^2 p}{d} \nonumber,
\end{eqnarray} 
where the middle sum is a Weil sum, see Theorem \ref{thm4466.430}, and $(p-1)/q\leq 1$. Replacing \eqref{eq3344.506} into \eqref{eq3344.504}, and rearranging it, yield
\begin{eqnarray}\label{eq3344.508}
\rho(d,p)&=&\sum_{\gcd(n,(p-1)/d)=1} e^{ \frac{i2\pi a \tau^{dn}}{p}} \\
&=& \sum_{ n \leq (p-1)/d}\frac{1}{q} \sum_{1 \leq t\leq q-1,}  \sum_{1 \leq s\leq p-1} \omega^{t(n-s)}e^{\frac{i2\pi a \tau^{ds}}{p}} \sum_{\substack{r \mid (p-1)/d \\ r \mid n}}\mu(d) \nonumber\\
&&\hskip 3.0 in +O\left(\frac{2p^{1/2}\log^2 p}{d}\right ) \nonumber \\
&=&\frac{1}{q} \sum_{1 \leq t\leq q-1} \left ( \sum_{1 \leq s\leq p-1} \omega^{-ts}e^{\frac{i2\pi a \tau^{ds}}{p}}\right )\left (\sum_{r \mid (p-1)/d} \mu(d) \sum_{\substack{n \leq (p-1)/d, \\ r \mid n}}   \omega^{tn} \right ) \nonumber\\
&&\hskip 3.0 in +O\left(\frac{2p^{1/2}\log^2 p}{d}\right ) \nonumber.
\end{eqnarray} 
Taking absolute value, and applying Lemma \ref{lem3366.200} to the inner sum, and Theorem \ref{thm4466.430} to the middle sum, yield
\begin{eqnarray} \label{eq3344.510}
|\rho(d,p)|&=& \left | \sum_{ \gcd(n, (p-1)/d)=1} e^{\frac{i2\pi a \tau^{dn}}{p}} \right | \\
&\leq&\frac{1}{q}  \sum_{1 \leq t\leq q-1} \left | \sum_{1 \leq s\leq p-1} \omega^{-ts}e^{i2\pi a \tau^{ds}/p} \right | \cdot  \left |\sum_{r \mid (p-1)/d} \mu(d) \sum_{\substack{n \leq (p-1)/d, \\ r \mid n}}   \omega^{tn} \right | \nonumber \\
&&\hskip 3.0 in+O\left(\frac{2p^{1/2}\log^2 p}{d}\right )\nonumber \\
&\ll&\frac{1}{q}  \sum_{1 \leq t\leq q-1} \left ( 2p^{1/2} \log^2 p \right ) \cdot  \left ( \frac{4q \log \log p}{\pi t} \right )+O\left(\frac{2p^{1/2}\log^2 p}{d}\right )\nonumber\\
&\ll& p^{1/2} \log^3 p \nonumber.
\end{eqnarray}
The last summation in \eqref{eq3344.510} uses the estimate 
\begin{equation} \label{eq3344.512} \sum_{1 \leq t\leq q-1}\frac{1}{t}\ll \log q\ll \log p
\end{equation} since $q=p+o(p)$. This is restated in the simpler notation $p^{1/2}\log ^3 p  \leq p^{1-\varepsilon}$ for any arbitrary small number $\varepsilon \in (0,1/2)$, as $x \to \infty$. 
\end{proof}

The upper bound given in Theorem \ref{thm3344.500} for maximal $\varepsilon <1/2$ seems to be optimum. A different proof, which has a weaker upper bound is included here as a reference for a second independent proof.
\begin{thm}  \label{thm3344.550}  {\normalfont (\cite[Theorem 6]{FS00})} Let \(p\geq 2\) be a large prime, and let $\tau $ be a primitive root modulo $p$. Then,
\begin{equation}
\sum_{\gcd(n,p-1)=1} e^{i2\pi a \tau^n/p} \ll  p^{1-\varepsilon} 
\end{equation} 
for any integer $a \in [1, p-1]$, and any arbitrary small number $\varepsilon >0$ is a small number. 
\end{thm}
Other related results are given in \cite{CC09}, \cite{FS01}, \cite{GZ05}, and \cite[Theorem 1]{GK05}.

\section{Explicit Exponential Sums} \label{S3535}
An explicit version of Theorem \ref{thm3344.500} and Theorem \ref{thm3344.550} is computed below.

\begin{thm} \label{thm3535.100} Let $m\geq 1$ be an integer, and let $Q\geq 1$ be the period of the element $w\in \Z/m\Z$. If the number $P<Q$, then 
\begin{equation}\label{eq3535.100}
\sum_{1\leq n \leq P}e^{i2\pi aw^{n}/m}\leq c_0P^{1-\varepsilon}.
\end{equation}
where $a\ne0$, $\varepsilon=c_1\log P/\log m<1$ is a small number, and $c_0,c_1>0$ are constants.
\end{thm}
\begin{proof} A discussion of this exponential sum and a proof appears in \cite[p.\ 8]{KN92}.
\end{proof}

The complete exponential sum
\begin{equation}\label{eq3535.102}
\sum_{1\leq n \leq Q}e^{i2\pi aw^{n}/m}
\end{equation}
is known to be a very small number or to vanish. An upper bound for a related and different exponential sum will be used in the analysis of the orders of elements in finite rings.
\begin{thm} \label{thm3535.200} Let $m\geq 1$ be an integer, and let $Q\geq 1$ be the period of the element $w\in \Z/m\Z$. If the number $P<Q$, then 
\begin{equation}\label{eq3535.100}
\sum_{\substack{1\leq n \leq P\\ \gcd(n,\varphi(m))}}e^{i2\pi aw^{n}/m}\leq c_0m^{\varepsilon}P^{1-2\varepsilon}.
\end{equation}
where $a\ne0$, $2\varepsilon=c_1\log P/\log m<1$ is a small number, and $c_0,c_1>0$ are constants.
\end{thm}
\begin{proof} Let $P=Q-1$, and rewrite the exponential sum in the form
\begin{eqnarray}\label{eq3535.105}
\sum_{\substack{1\leq n \leq P\\ \gcd(n,\varphi(m)}}e^{i2\pi aw^{n}/m}
&=& \sum_{1\leq n \leq P}e^{i2\pi aw^{n}/m} \sum_{\substack{d\mid n\\ d\mid \varphi(m)}}\mu(d)\\
&=& \sum_{ d\mid \varphi(m)}\mu(d)\sum_{\substack{1\leq n \leq P\\d\mid n}}e^{i2\pi aw^{n}/m} \nonumber.
\end{eqnarray}
Taking absolute value, and applying Theorem \ref{thm3535.100} to the inner exponential sum return
\begin{eqnarray}\label{eq3535.105}
\left |\sum_{\substack{1\leq n \leq P\\ \gcd(n,\varphi(m))}}e^{i2\pi aw^{n}/m}\right |
&\leq& \sum_{ d\mid \varphi(m)}1\left |\sum_{\substack{1\leq n \leq P\\d\mid n}}e^{i2\pi aw^{n}/m}\right |\\
&\leq& \sum_{ d\mid \varphi(m)}1\cdot c_1\left (\frac{P}{d}\right)^{1-2\varepsilon} \nonumber\\
&\leq& c_2P^{1-2\varepsilon}\sum_{ d\mid \varphi(m)}\frac{1}{d^{1-2\varepsilon}} \nonumber\\
&\leq& c_3P^{1-2\varepsilon}\sum_{ d\mid \varphi(m)}1 \nonumber\\
&\leq& c_4P^{1-2\varepsilon}\varphi(m)^{\varepsilon} \nonumber,
\end{eqnarray}
where $c_0=c_4, c_1, c_2, c_3>0$ are constants.  Plugging the trivial upper bounds $P\leq m$, and $\varphi(m)\leq m$ complete the verification of the inequality 
\begin{eqnarray}\label{eq3535.109}
\left |\sum_{\substack{1\leq n \leq P\\ \gcd(n,\varphi(m))}}e^{i2\pi aw^{n}/m}\right |
&\leq& c_4P^{1-2\varepsilon}\varphi(m)^{\varepsilon} \leq c_4m^{1-\varepsilon},
\end{eqnarray}
for sufficiently large $m$.
\end{proof}

\section{Equivalent Exponential Sums} \label{S3500}
This section demonstrate that the exponential sums 
\begin{equation} \label{eq3500.503}
\sum_{\gcd(n,(p-1)/d)=1} e^{i2\pi a \tau^{dn}/p} \quad \text{ and } \quad \sum_{\gcd(n,(p-1)/d)=1} e^{i2\pi \tau^{dn}/p},
\end{equation}
where $d\mid p-1$, and $a\ne0$, are asymptotically equivalent. This result expresses this exponential sum as a sum of simpler exponential sum and an error term. The proof is entirely based on established results and elementary techniques.  

\begin{thm}   \label{thm3500.900}  Let \(p\geq 2\) be a large primes, and let $d\mid p-1$ be a small divisor. If $\tau $ be a primitive root modulo $p$, then,
\begin{equation} \sum_{\gcd(n,(p-1)/d)=1} e^{i2\pi a \tau^{dn}/p} = \sum_{\gcd(n,(p-1)/d)=1} e^{i2\pi  \tau^{dn}/p} + O(p^{1/2} \log^4 p),
\end{equation} 
for any integer $ a \in [1, p-1]$. 	
\end{thm} 
\begin{proof} For any integer $a\geq 1$, the exponential sum has the representation 
\begin{eqnarray} \label{eq3500.902}
\rho(a,d,p)&=& 	\sum_{\gcd(n,(p-1)/d)=1} e^{\frac{i2\pi a \tau^{dn}}{p}} \\
&=&\frac{1}{q} \sum_{1 \leq t\leq q-1} \left ( \sum_{1 \leq s\leq p-1} \omega^{-ts}e^{\frac{i2\pi a\tau^{ds}}{p}}\right )\left (\sum_{r \mid (p-1)/d} \mu(r) \sum_{\substack{n \leq (p-1)/d, \\ r \mid n}}   \omega^{tn} \right ) \nonumber\\
&&\hskip 3.0 in +O\left(\frac{2p^{1/2}\log^2 p}{d}\right )\nonumber,
\end{eqnarray} 
confer equations \eqref{eq3344.502} to \eqref{eq3344.508} for details. And, for $a=1$, 
\begin{eqnarray} \label{eq3500.904}
\rho(1,d,p)&=& 	\sum_{\gcd(n,(p-1)/d)=1} e^{\frac{i2\pi  \tau^{dn}}{p}} \\
&=& \frac{1}{q} \sum_{1 \leq t\leq q-1} \left ( \sum_{1 \leq s\leq p-1} \omega^{-ts}e^{\frac{i2\pi  \tau^{ds}}{p}}\right )\left (\sum_{r \mid (p-1)/d} \mu(r) \sum_{\substack{n \leq (p-1)/d, \\ r \mid n}}   \omega^{tn} \right )\nonumber\\
&&\hskip 3.0 in +O\left(\frac{2p^{1/2}\log^2 p}{d}\right )\nonumber,
\end{eqnarray}
respectively, see equations \eqref{eq3344.502} to \eqref{eq3344.508}. Differencing (\ref{eq3500.902}) and (\ref{eq3500.904}) produces 
\begin{eqnarray} \label{eq3355.390}
D(d,p)&=&	\sum_{\gcd(n,(p-1)/d)=1} e^{i2\pi a \tau^{dn}/p} -\sum_{\gcd(n,(p-1)/d)=1} e^{i2\pi  \tau^{dn}/p} \\
&=&     \frac{1}{q} \sum_{0 \leq t\leq q-1} \left ( \sum_{1 \leq s\leq p-1} \omega^{-ts}e^{\frac{i2\pi a \tau^{ds}}{p}}-\sum_{1 \leq s\leq p-1} \omega^{-ts}e^{\frac{i2\pi  \tau^{ds}}{p}}\right ) \nonumber \\
&& \hskip 2.30 in \times \left (\sum_{r \mid (p-1)/d} \mu(r) \sum_{\substack{n \leq (p-1)/d, \\ r \mid n}}   \omega^{tn} \right ) \nonumber.
\end{eqnarray}
By  Lemma \ref{lem3366.200}, the relatively prime summation kernel is bounded by
\begin{eqnarray} \label{eq3355.393}
 \left |\sum_{r \mid (p-1)/d} \mu(r) \sum_{\substack{n \leq (p-1)/d, \\ r \mid n}}   \omega^{tn} \right | 
&=& \left | \sum_{\gcd(n, (p-1)/d)=1}\omega^{tn} \right | \nonumber \\ 
&\leq &  \frac{4 q \log \log p} {\pi t}, 
\end{eqnarray}
and by Theorem \ref{thm4466.430}, the difference of two Weil sums (or Gauss sums) is bounded by
\begin{eqnarray} \label{eq3355.395}
\delta(d,p)&=&  \left | \sum_{1 \leq s\leq p-1} \omega^{-ts}e^{\frac{i2\pi a \tau^{ds}}{p}}-\sum_{1 \leq s\leq p-1} \omega^{-ts}e^{\frac{i2\pi  \tau^{ds}}{p}}\right  |  \nonumber \\
&=& \left |  \sum_{1 \leq s\leq p-1} \chi(s) \psi_a(s) - \sum_{1 \leq s\leq p-1} \chi(s) \psi_1(s) \right | \nonumber \\ 
&\leq & 4 p^{1/2} \log^2 p, 
\end{eqnarray}
where  $\chi(s)=e^{i \pi s t/p}$, and $ \psi_a(s)=e^{i2\pi a \tau^{ds}/p}$. Taking absolute value in (\ref{eq3355.390}) and replacing \eqref{eq3355.393}, and \eqref{eq3355.395}, return
\begin{eqnarray} \label{eq3355.397}
|D(d,p)|&=& \left|  	\sum_{\gcd(n,p-1)=1} e^{i2\pi b \tau^n/p} -\sum_{\gcd(n,p-1)=1} e^{i2\pi  \tau^n/p} \right| \\ 
& \leq &      \frac{1}{q} \sum_{0 \leq t\leq q-1} \left ( 4p^{1/2} \log^2 p \right ) \cdot \left ( \frac{4 q \log \log p} {t} \right ) \nonumber\\
&\leq & 16p^{1/2} (\log^2 p)(\log q)( \log \log p )\nonumber\\
&\leq & 16p^{1/2} \log^4 p \nonumber,
\end{eqnarray}
where $q=p+o(p)$. 
\end{proof}

\section{Double Exponential Sums} \label{S3700}

\begin{lem} \label{lem3700.200} Given a small number $\varepsilon>0$. Let $p$ be a large prime number, and let $d\mid p-1$ be a small divisor. If $\tau$ is a primitive root modulo $p$, then,
\begin{equation}\label{eq3700.200}
\sum_{\substack{0< a< p\\\gcd(n,\varphi(p)/d)=1}}e^{\frac{i2\pi a(\tau^{dn}-u)}{p}}
\ll p^{1-\varepsilon}.
\end{equation}
\end{lem}
\begin{proof} To compute an upper bound, rearrange the double finite sum as follows.
\begin{equation}\label{eq2777.204}
\sum_{\substack{0< a< p\\\gcd(n,\varphi(p)/d)=1}}e^{\frac{i2\pi a(\tau^{dn}-u)}{p}}=\sum_{0< a< p}e^{-i2\pi au/p}\sum_{\gcd(n,\varphi(p)/d)=1}e^{i2\pi a\tau^{dn}/p}.
\end{equation}
Applying Theorem \ref{thm3500.900} to remove the $a$ dependence of the inner finite sum, yields
\begin{eqnarray}\label{eq2777.106}
T(p)
&=& \sum_{0< a< p}e^{-i2\pi au/p}\sum_{\gcd(n,\varphi(p)/d)=1}e^{i2\pi a\tau^{dn}/p}\\
&=&\sum_{0< a< p}e^{-i2\pi au/p}\left (\sum_{\gcd(n,\varphi(p)/d)=1}e^{i2\pi \tau^{dn}/p} +O\left (p^{1/2}\log^4 p \right )\right )\nonumber.
\end{eqnarray}
Taking absolute value, and  applying Theorem \ref{thm3344.500} (or Theorem \ref{thm3535.100}), yield
\begin{eqnarray}\label{eq7777.002}
|T(p)|
&\leq& \left |\sum_{0< a< p}e^{-i2\pi au/p}\right |\left |\sum_{\gcd(n,\varphi(p)/d)=1}e^{i2\pi \tau^{dn}/p} +O\left (p^{1/2}\log^4 p \right )\right |\nonumber\\
&\leq& \left |\sum_{0< a< p}e^{-i2\pi au/p}\right |\left (\left |\sum_{\gcd(n,\varphi(p)/d)=1}e^{i2\pi \tau^{dn}/p}\right | +O\left (p^{1/2}\log^4 p \right )\right )\nonumber\\
&\ll& \left | -1 \right | \cdot \left (p^{1-\varepsilon}+p^{1/2}\log^4 p\right )\nonumber\\
&\ll& p^{1-\varepsilon},
\end{eqnarray}
where $\sum_{0< a< p}e^{i2\pi au/p}=-1$ for any $u\ne0$, and $\varepsilon<1/2$ is a small number. 
\end{proof}

\section{The Main Term} \label{S2888}
The main term of two or more simultaneous elements requires the average order of a product of totient functions. A lower bound for the product two totients will be computed here.
\begin{lem} \label{lem2888.100} If $x\geq 1$ is a large number, and $d,e\ll (\log x)^B$, with $B\geq 0$, then
\begin{equation}\label{eq2888.100}
\sum_{x\leq p\leq 2x} \frac{\varphi((p-1)/d)}{p}\cdot\frac{\varphi((p-1)/e)}{p}\gg \frac{x}{(\log x)^{4B+1}(\log \log x)^2}.
\end{equation}
\end{lem}
\begin{proof} The totient function has the lower bound $\varphi(n)/n\gg1/\log\log n$, see \cite[Theorem 15]{RS62}. Replacing this estimate yields
\begin{eqnarray}\label{eq1777.102}
M(x,u,v)&=&\sum_{x\leq p\leq 2x} \frac{\varphi((p-1)/d)}{p}\cdot\frac{\varphi((p-1)/e)}{p}\\
&=&\sum_{x\leq p\leq 2x} \frac{\varphi((p-1)/d)}{p-1}\cdot\frac{\varphi((p-1)/e)}{p-1}\left (1-\frac{1}{p}\right )^2 \nonumber\\
&\gg&\sum_{\substack{x\leq p\leq 2x\\p\equiv 1 \bmod de}} \frac{1}{d}\frac{1}{\log \log p}\cdot \frac{1}{e}\frac{1}{\log \log p}\nonumber\\
&\gg& \frac{1}{(\log x)^{2B}}\frac{1}{(\log \log x)^2}\sum_{\substack{x\leq p\leq 2x\\p\equiv 1 \mod de}} 1\nonumber,
\end{eqnarray}
since $d,e\ll (\log x)^B$, with $B\geq 0$. Applying the \textit{prime number theorem on arithmetic progression} over the short interval $[x,2x]$, yields
\begin{eqnarray}\label{eq1777.104}
M(x,u,v)&=&\frac{1}{(\log x)^{2B}}\frac{1}{(\log \log x)^2}\sum_{\substack{x\leq p\leq 2x\\p\equiv 1 \bmod de}} 1\\
&\gg& \frac{x}{(\log x)^{2B}(\log \log x)^2}\cdot \left (\frac{x}{\varphi(de)\log x}+O\left (xe^{-c\sqrt{\log x}}\right)\right) \nonumber\\
&\gg& \frac{x}{(\log x)^{4B+1}(\log \log x)^2}\nonumber,
\end{eqnarray}
where $\varphi(de)\leq de\leq (\log x)^{2B}$, and $c> 0$ is an absolute constant.
\end{proof}

This analysis is effective and unconditional for the prescribed indices $d,e\ll \log^B x$, where $B\geq 0$ is a constant, as limited by the current version of the prime number theorem on arithmetic progressions, confer \cite[Theorem 3.10]{EL85}. \\

The exact asymptotic for the average order of a product $k$ totient functions over the primes is proved in \cite{VR73}, and related discussions are given in \cite[p.\ 16]{MP04}. The generalization to number fields appears in \cite{HG84}. However, the exact asymptotic for the average order of a product of $k$ totient functions over the primes in arithmetic progressions seems to be unknown. For example, for equal prescribed multiplicative orders $(p-1)/d$, it should have the form

\begin{equation}\label{eq2888.110}
\sum_{\substack{x\leq p\leq 2x\\p\equiv a \bmod q}}\left (\frac{\varphi((p-1)/d)}{p-1}\right )^k\stackrel{?}{=}A_1\frac{x}{\varphi(q)\log x}+O\left (xe^{-c\sqrt{\log x}}\right),
\end{equation}
where $q=d^k\leq (\log x)^{Bk}$, $B\geq 0$ constant, $1\leq a<q$ are relatively prime integers, $A_1=A_1(a,q)>0$ is a constant, and $c>0$ is an absolute constant. But, for distinct prescribed multiplicative orders $(p-1)/d_i$, it should have the form
\begin{equation}\label{eq2888.115}
\sum_{\substack{x\leq p\leq 2x\\p\equiv a \bmod q}}\prod_{1\leq i\leq k}\left (\frac{\varphi((p-1)/d_i)}{p-1}\right )\stackrel{?}{=}A_k\frac{x}{\varphi(q)\log x}+O\left (xe^{-c\sqrt{\log x}}\right),
\end{equation}
where $q=d_1d_2\cdots d_k\leq (\log x)^{Bk}$, and $A_k=A_k(a,q)>0$, which is slightly more complex. 

\section{The Error Terms} \label{S2777}

\begin{lem} \label{lem2777.200} Assume $\ord_pu \ne (p-1)/d$. If $x$ is a large number, then
\begin{equation}\label{eq2777.200}
E_1(x)=\sum_{x\leq p\leq 2x} \frac{1}{p}\sum_{\substack{0\leq  a< n\\\gcd(n,\varphi(p)/d)=1}}e^{\frac{i2\pi a(\tau^{dn}-u)}{p}} \cdot \frac{1}{p} \sum_{\substack{0<leq b< \varphi(p)/e\\\gcd(m,\varphi(p)/e)=1}}e^{\frac{i2\pi b(\tau^{em}-v)}{p}}
=0\nonumber.
\end{equation}
\end{lem}
\begin{proof} By hypothesis, $\ord_pu \ne (p-1)/d$, so the first finite sum
\begin{equation}\label{eq2777.202}
\sum_{\substack{0\leq  a< n\\\gcd(n,\varphi(p)/d)=1}}e^{\frac{i2\pi a(\tau^{dn}-u)}{p}} 
=0
\end{equation}
vanishes, see Lemma \ref{lem3000.114}.
\end{proof}

\begin{lem} \label{lem2777.300} Assume $\ord_p v \ne (p-1)/e$. If $x$ is a large number, then
\begin{equation}\label{eq2777.300}
E_2(x)=\sum_{x\leq p\leq 2x} \frac{1}{p}\sum_{\substack{0< b< n\\\gcd(m,\varphi(p)/e)=1}}e^{\frac{i2\pi b(\tau^{em}-v)}{p}} \cdot \frac{1}{p} \sum_{\substack{0\leq b< \varphi(p)/e\\\gcd(m,\varphi(p)/e)=1}}e^{\frac{i2\pi b(\tau^{em}-v)}{p}}
=0\nonumber.
\end{equation}
\end{lem}
\begin{proof} By hypothesis, $\ord_pv \ne (p-1)/e$, so the second finite sum
\begin{equation}\label{eq2777.302}
\sum_{\substack{0\leq  b< n\\\gcd(m,\varphi(p)/e)=1}}e^{\frac{i2\pi b(\tau^{em}-v)}{p}} 
=0
\end{equation}
vanishes, see Lemma \ref{lem3000.114}.
\end{proof}

\begin{lem} \label{lem2777.400} If $x$ is a large number, then
\begin{equation}\label{eq2777.400}
E_3(x)=\sum_{x\leq p\leq 2x} \frac{1}{p}\sum_{\substack{0< a< p\\\gcd(n,\varphi(p)/d)=1}}e^{\frac{i2\pi a(\tau^{dn}-u)}{p}} \cdot \frac{1}{p} \sum_{\substack{0< b< \varphi(p)/e\\\gcd(m,\varphi(p)/e)=1}}e^{\frac{i2\pi b(\tau^{em}-v)}{p}}
\ll x^{1-2\varepsilon}\nonumber.
\end{equation}
\end{lem}
\begin{proof} To compute an upper bound, define the exponential sum
\begin{equation}\label{eq2777.402}
T(d,p)=\sum_{\substack{0< a< p\\\gcd(n,\varphi(p)/d)=1}}e^{\frac{i2\pi a(\tau^{dn}-u)}{p}}.
\end{equation}
Now, apply Lemma \ref{lem3700.200} to each factor $T(d,p)$ and $T(e,p)$ to obtain the followings.
\begin{eqnarray}\label{eq2777.404}
E_3(x)&=&\sum_{x\leq p\leq 2x} \left (\frac{1}{p}\cdot T(d,p) \right )\cdot \left (\frac{1}{p}\cdot T(e,p)\right )\\
&\ll&\sum_{x\leq p\leq 2x} \left (\frac{1}{p}\cdot p^{1-\varepsilon} \right )\cdot \left (\frac{1}{p}\cdot p^{1-\varepsilon}\right ) \nonumber.
\end{eqnarray}
Take an upper bound, and apply the prime number theorem:  
\begin{eqnarray}\label{eq2777.406}
E_3(x)
&\ll& \frac{1}{x^{2\varepsilon}}\sum_{x\leq p\leq 2x} 1\\
&\ll& x^{1-2\varepsilon}\nonumber,
\end{eqnarray}
as $x\to \infty$.
\end{proof}

\section{Simultaneous Prescribed Multiplicative Orders} \label{S3388}
The \textit{multiplicative order} of an element $u\in \F_p$ in finite field is the smallest integer $n\geq 1$ for which $u^n=1$ in $\F_p$, see Section \ref{S3000} for additional details. The simpler case of simultaneous prescribed multiplicative orders of two admissible rational numbers is investigated in this section. The analysis of simultaneous and prescribed multiplicative orders for $k$-tuple of admissible rational numbers are similar to this analysis, but have bulky and cumbersome notation.   

\begin{dfn}\label{dfn3388.000} {\normalfont Fix a pair of rational numbers $u,v\ne\pm1$ such that $u^av^b\ne\pm1$ for any $a,b\in \Z$. The elements $u, v\in \F_p$ are said to be simultaneous of equal orders $\ord_pu \mid p-1$, and $\ord_pv \mid p-1$ respectively, if $\ord_pu=\ord_pv$ infinitely often as $p\to \infty$. Otherwise, the elements are said to be simultaneous of unequal orders, if $\ord_pu\ne \ord_pv$, infinitely often as $p\to \infty$.  
}
\end{dfn}

Given a pair of small integer indices $d\geq1$ and $e\geq1$, the number of primes $x\leq p\leq 2x$, which have simultaneous elements $u>1$ and $v>1$ of prescribed orders $\ord_pu \mid (p-1)/d$ and $\ord_pv \mid (p-1)/e$ modulo $p\geq 2$, respectively, is defined by
\begin{eqnarray}\label{eq3388.000}
R(x,u,v)&=&\#\left \{\; x\leq p\leq 2x: p \text{ is prime and }\right.\\
 &&\hskip 1.0 in \left. \ord_p(u)=(p-1)/d, \ord_{p}(v)=(p-1)/e \;\right \}\nonumber, 
\end{eqnarray}
The small integers indices $d\geq1$ and $e\geq1$ prescribed the multiplicative orders of the fixed pair $u>1$ and $v>1$

\begin{proof} (Theorem \ref{thm8888.002}) Substitute the indicator function $\Psi_{p}(u,d)$ for elements $u$ of order $\ord_pu=(p-1)/d$ modulo $p$, and the indicator function $\Psi_{p}(u,e)$ for elements $v$ of order $\ord_pv=(p-1)/e$ modulo $p$, see Lemma \ref{lem3000.114}, to construct the associated counting function for the number of such primes $p\in [x,2x]$:
\begin{eqnarray}\label{eq8888.002}
R(x,u,v)&=&\sum_{x\leq p\leq 2x} \Psi_{p}(u,d)\cdot \Psi_{p}(v,e)\\
&=&\sum_{x\leq p\leq 2x}  \frac{1}{p}\sum_{\substack{0\leq a< p\\\gcd(n,\varphi(p)/d)=1}}e^{i2\pi a(\tau^{dn}-u)/p}\cdot \frac{1}{p} \sum_{\substack{0\leq b< p\\\gcd(m,\varphi(p)/e)=1}}e^{i2\pi b(\tau^{em}-v)/p}\nonumber\\
&=&M(x,u,v)\quad +\quad E_1(x)+\quad E_2(x)+\quad E_3(x)\nonumber.
\end{eqnarray}
The main term $M(x,u,v)$ is determined by $(a,b)=(0,0)$, the error terms $E_1(x)$, $E_2(x)$, and $E_3(x)$ are determined by $(a,b)=(0,b\ne0)$, $(a,b)=(a\ne0,0)$, and $(a,b)\ne(0,0)$, respectively. \\

Summing the main term $M(x,u,v)$, estimated in Lemma \ref{lem2888.100}, and the error terms $E_1(x)$, $E_2(x)$, and $E_3(x)$ estimated in Lemma \ref{lem2777.200}, Lemma \ref{lem2777.300}, and Lemma \ref{lem2777.400} respectively, yield  
\begin{eqnarray}\label{eq7777.002}
R(x,u,v)&=&M(x,u,v)\quad +\quad E_1(x)+\quad E_2(x)+\quad E_3(x)\\
&\gg&\frac{x}{(\log x)^{4B+1}(\log \log x)^2}+ 0+0+x^{1-2\varepsilon}\nonumber\\
&\gg & \frac{x}{(\log x)^{4B+1}(\log \log x)^2}  \nonumber,
\end{eqnarray}
where $\varepsilon>0$ is a small number, as $x\to\infty$.
\end{proof}
The current analysis is unconditional for any indices product $de\ll (\log x)^{2B}$, where $B\geq 0$ is a constant. Assuming the RH, it appears that this analysis can handled any indices product $de$ as large as $de\leq p^{1/2-\delta}$, where $\delta>0$, but no effort was made to verify this observation, confer the proof for the lower bound of the main term in Section \ref{S2888} for some information. \\

\section{Probabilistic Results For Simultaneous Orders} \label{S6688}
\begin{thm}\label{thm6688.000} For any pair of random relatively prime integers $a>1$ and $b>1$, and any sufficiently large prime $p$, the ratio 
\begin{equation}\label{eq6688.100}
\frac{\ord_p(a)}{\ord_p(b)}\ne1\end{equation}
is true with probability $1+o(1)> 1/2$.
\end{thm} 

\begin{proof} Fix a large prime $p$. Two random relatively prime integers $a$ and $b$ have the same multiplicative order if and only if $\ord_p(a)=\ord_p(b)= (p-1)/d_0$ for at least one divisor $d_0 \mid p-1$. Let $\alpha_2> 0$ denotes the probability that 
\begin{equation}\label{eq6688.105}
\frac{\ord_p(a)}{\ord_p(b)}=1
\end{equation}
is true. Otherwise, $\alpha_2= 0$, and the claim in \eqref{eq6688.100} is trivially true. Assuming statistical pseudo independence, this event occurs with probability
\begin{eqnarray}\label{eq6688.110}
c_0\frac{1}{p-1}\cdot \frac{1}{p-1}<\alpha_2
&=&c_0\frac{\varphi((p-1)/d_0)}{p-1}\cdot \frac{\varphi((p-1)/d_0)}{p-1}\\
&<&\sum_{d\mid p-1}
c_d\left(\frac{\varphi((p-1)/d)}{p-1}\right)^2\nonumber\\
&\leq& \sum_{d\mid p-1}\frac{\varphi((p-1)/d)}{p-1}\nonumber\\
&=&\frac{\varphi(p-1)}{p-1}\leq \frac{1}{2}\nonumber,
\end{eqnarray}
where $c_i=c_i(a,b)\leq 1$ is a statistical independence correction factor. The lower bound $1/(p-1)^2$ in \eqref{eq6688.110} follows from the hypothesis $a>1$ and $b>1$. Hence, random relatively prime integers $a$ and $b$ of multiplicative order $\ord_p(a)\ne\ord_p(b)$, satisfy the ratio 

\begin{equation}\label{eq6688.120}
\frac{\ord_p(a)}{\ord_p(b)}\ne1
\end{equation}
with probability
\begin{equation}1-\alpha_2> \frac{1}{2}.\end{equation}
\end{proof}



\currfilename.\\

\end{document}